\documentclass[12pt]{iopart}

\usepackage{iopams}  
\usepackage{amsthm}

\usepackage{graphicx} 
\usepackage{epstopdf}

\newcommand{\Z}{\mathbb{Z}}

\newcommand{\R}{\mathbb{R}}
\newcommand{\C}{\mathbb{C}}

\newcommand*{\ra}{\rightarrow}

\newtheoremstyle{example}{\topsep}{\topsep}%
     {\itshape}%         Body font
     {}%         Indent amount (empty = no indent, \parindent = para indent)
     {\bfseries}% Thm head font
     {.}%        Punctuation after thm head
     {\newline}%     Space after thm head (\newline = linebreak)
     {\thmname{#1}\thmnumber{ #2}\thmnote{ #3}}%         Thm head spec

\theoremstyle{example}
\newtheorem{theorem}{Theorem}[section]
\newtheorem{proposition}[theorem]{Proposition}
\newtheorem{lemma}[theorem]{Lemma}

\newtheorem{definition}[theorem]{Definition}

\newtheoremstyle{remark}{\topsep}{\topsep}%
     {}%         Body font
     {}%         Indent amount (empty = no indent, \parindent = para indent)
     {\bfseries}% Thm head font
     {.}%        Punctuation after thm head
     {\newline}%     Space after thm head (\newline = linebreak)
     {\thmname{#1}\thmnumber{ #2}\thmnote{ #3}}%         Thm head spec

\theoremstyle{remark}
\newtheorem{remark}[theorem]{Remark}

\newtheorem{example}[theorem]{Example}

\newtheorem{corollary}[theorem]{Corollary}

\begin{document}

\title[Recovering quantum graphs from their Bloch spectrum]{Recovering quantum graphs from their Bloch spectrum}

\author{Ralf Rueckriemen}

\address{Dartmouth College,
Kemeny Hall,
Hanover, 03755 NH, USA}
\ead{ralf.rueckriemen@dartmouth.edu}
\begin{abstract}
 We define the Bloch spectrum of a quantum graph to be the collection of the spectra of a family of Schr\"odinger operators 
parametrized by the cohomology of the quantum graph. We show that the 
Bloch spectrum determines the Albanese torus, the block structure and the planarity of the graph. It determines a geometric dual of a 
planar graph. This enables us to show that the Bloch spectrum completely determines planar $3$-connected quantum graphs.
\end{abstract}

\vspace{8cm}
I would like to thank my supervisor Carolyn Gordon for her extensive support. Without our frequent meetings and her numerous helpful and 
necessary suggestions this paper could not have been written.

\maketitle

\section{introduction}

We consider finite combinatorial graphs in which each edge is equipped with a positive finite length, often called metric graphs. 
We define a Schr\"odinger operator on a metric graph. The pair of a metric graph together with 
a Schr\"odinger operator is called a quantum graph.

Quantum graphs are studied in mathematics and physics. They serve as simplified models in many settings involving 
wave propagation. The fact that they are essentially one-dimensional makes explicit computations possible in various situations. On the other hand 
the graph structure gives them enough complexity to be useful models. The papers \cite{Kuchment08} and \cite{BolteEndres08} provide an 
introduction and a survey of quantum graphs and the trace formulae that are often used to study them. 

A function on a quantum graph consists of a function on each edge, where the edges are viewed as intervals.  A Schr\"odinger operator acts 
on the space of all functions that are smooth on each edge and satisfy specified conditions at the vertices.   We will impose the Kirchoff 
vertex conditions, which require that the function be continuous and that the sum of the inward pointing derivatives on all edges incident 
at the vertex be zero.  Kirchoff conditions model a conservation of flow.   A Schr\"odinger operator is a second order operator on each 
edge with leading term the standard Laplacian $-\left(\frac{\partial}{\partial x}\right)^2$. The first order part is called the magnetic 
potential.

The question to what degree the spectrum of the Laplace operator determines the underlying space was popularized my Kac in \cite{Kac66} 
in the manifold setting. The Schr\"odinger operator with zero magnetic potential is the standard Laplacian 
on a quantum graph. The relation between its spectrum or more generally the spectrum of a Schr\"odinger operator and the underlying graph 
 is an active area of research. Various exact trace formulae relate the two, see for example \cite{Roth83},\cite{KottosSmilansky99} or 
\cite{BolteEndres08}. 

A quantum graph is determined by the 
spectrum of a single Schr\"odinger operator,  if some genericity assumptions on the edge lengths in the 
quantum graph are made, \cite{GutkinSmilansky01}. 
This is not true without a genericity assumption on the edge lengths. Various examples of isospectral 
non-isomorphic quantum graphs exist, see for example \cite{vonBelow99}, \cite{GutkinSmilansky01} or \cite{BPB09}. The last also proves a 
generalization of Sunada's theorem to construct isospectral quantum graphs. 

The idea of this paper is to look at the spectra of an entire 
collection of Schr\"odinger operators, which we call the Bloch spectrum of a quantum graph. 

The classical Bloch spectrum of a torus $\R^n/L$ assigns to each character $\chi:L\ra \C^*$ the spectrum of the Euclidean Laplacian acting on the 
space of functions on $\R^n$ that satisfy $f(x+l)=\chi(l)f(x)$ for all $l\in L$.   See, for example, \cite{ERT84} for inverse spectral results 
concerning the Bloch spectrum.  As pointed out by Guillemin \cite{Guillemin90}, the Bloch spectrum can also be interpreted as the collection of 
spectra of all operators $\nabla^*\nabla$ acting on $C^\infty(\R^n/L)$, where $\nabla$ is a connection with zero curvature. The set of these 
connections is given by 
$\nabla=(d+i\alpha)$ with $\alpha$ a harmonic 1-form on $\R^n/L$. 
(One may take $\alpha$ to be any closed 1-form, but the spectrum depends only on the cohomology class of $\alpha$, so one may always assume 
$\alpha$ to be harmonic.) The correspondence with the classical notion is given by the association of the character 
$\chi(l)=e^{2\pi i\alpha(l)}$ to the harmonic form $\alpha$, where now $\alpha$ is viewed as a linear functional on $\R^n$.  The operators 
$\nabla =(d+i\alpha)$ with $\alpha$ closed may be viewed as the collection of all flat connections on the trivial Hermitian line bundle 
over the torus $\R^n/L$.   This interpretation of the Bloch spectrum admits a generalization to arbitrary Hermitian line bundles over 
a torus, where now one considers all connections with, say, harmonic curvature, see \cite{GGKW08}. 

Both interpretations of the Bloch spectrum can be carried over to quantum graphs. We use differential forms to define our operators, 
our approach is similar to the one in  \cite{Post09}.  We will consider operators of the form 
$\Delta_{\alpha}=(d+2\pi i\alpha)^*(d+2\pi i\alpha)$ and vary the $1$-form $\alpha$. Similarly to the setting of flat tori the spectrum 
depends only on the equivalence class of $\alpha$ in $H_{dR}^1(G,\R) \slash H_{dR}^1(G,\Z)$.
We can also define the Bloch spectrum using characters of the first fundamental group.
We show that these two notions of the Bloch spectrum of a quantum graph are equivalent.

We want to see what information about the quantum graph can be retrieved from the Bloch spectrum without any genericity assumptions 
on the quantum graph. 
It is known (see section seven for details) that the spectrum of the standard Laplacian $\Delta_0$ determines the dimension $n$ of 
$H^1(G,\R)$.  Thus from that spectrum alone, we know that $H^1(G,\R)/H^1(G,\Z)$ is isomorphic as a torus (i.e., as a Lie group) 
to $\R^n/ \Z^n$. 
Hence we can view the Bloch spectrum as a map that associates a spectrum to each $\alpha\in \R^n/ \Z^n$. 
We ask the following question:

\itshape
Suppose we are given a map that assigns a spectrum to each element $\alpha$ of $\R^n / \Z^n$ and we know these spectra form the Bloch 
spectrum of a finite quantum graph $G$. From this information, can one reconstruct $G$ both 
combinatorially and metrically?
\upshape

We will consider a generic $\alpha$, i.e., one whose orbit is dense in the torus $\R^n/\Z^n$, 
and we will just consider the spectra associated to an interval in the orbit of $\alpha \in \R^n/\Z^n$.

Our main results are as follows. 
\begin{theorem}
 The Bloch spectrum determines the Albanese torus, $Alb(G)=H_1(G,\R) \slash H_1(G,\Z)$, of a quantum graph as a Riemannian manifold. 
\end{theorem}
The Albanese torus contains information about the cycles in the graph and how they overlap. Note that the spectrum of a 
single Schr\"odinger operator does not determine the  
Albanese torus, there are examples of isospectral graphs with different Albanese tori, \cite{vonBelow99}. If the quantum graph is 
equilateral the Albanese torus also determines the complexity of the graph by a theorem in \cite{KotaniSunada00}.

The block structure of a graph contains the  broad structure of the graph, see definition \ref{define_block} for the 
definition of \emph{block structure}. 
\begin{theorem}
The Bloch spectrum determines the block structure of a quantum graph. 
\end{theorem}
The cycle space of a graph is closely related to its homology, we use some of its properties to show:
\begin{theorem}
The Bloch spectrum determines whether or not a graph is planar. 
\end{theorem}
Planarity is not determined by the spectrum of a single Schr\"odinger operator, \cite{vonBelow99}.
The information about the 
homology we read out from the Bloch spectrum allows us to construct a geometric dual of a planar quantum graph. We use it to show:
\begin{theorem}
Planar $3$-connected graphs are completely determined by their Bloch spectrum.
 \end{theorem}

The plan of this paper is as follows. In the second section we collect various facts about combinatorial graphs that will be needed later on. 
In the third section we define 
differential forms on quantum graphs and use them to define the Schr\"odinger operators and the Bloch spectrum via differential forms. 
We define the Bloch spectrum via characters in section 
four and show that this definition is equivalent to the differential form version. We then define the Albanese torus 
of a quantum graph in the fifth section. We discuss a trace formula for quantum graphs, the key tool that allows us to get
 information about 
the quantum graph from the Bloch spectrum. In section seven we show that the Bloch spectrum determines the length of a shortest 
representative of each element in $H_1(G,\Z)$, see theorem \ref{Bloch_H1}. This is the main theorem that relates the Bloch spectrum to the graph. The other 
theorems are just consequences from this one. We then show that the Bloch spectrum determines the Albanese torus of a quantum graph. 
In section eight we use these properties to show that the Bloch spectrum 
determines the block structure and planarity of a quantum graph. If the graph is planar it determines a geometric dual of the graph. 
This information completely determines the underlying combinatioral graph from the Bloch spectrum if the graph is planar and 3-connected. 
In section nine we show that if we 
know the underlying combinatorial graph and it is $3$-connected then the Bloch spectrum determines 
the length of all edges in the graph, so we can recover the full quantum graph. In section ten we will treat disconnected graphs and 
show that our results still hold in this case.

\section{Combinatorial graph theory}
%%%%%%%%%%%%%%%%%%%%%%%%%%%%%%%%%%%%%%%%%%%%%%%%%%%%%%%%%%%%%%%%%%%%%%%%%%%%%%%%%%%%%%%%%%%%%%%%%%%%%%%%%%%%%%%%%%%%%%%%%%%%%%%

This chapter collects various basic facts about combinatorial graphs that will be required later. The material is mostly taken from 
\cite{Diestel05}, which provides an excellent introduction to the area. 

All our graphs are finite and connected. We will treat the case of disconnected graphs in section ten. We allow loops and multiple edges. Let $G$ be a graph. We will denote the set of vertices by $V$ and 
the set of edges by $E(G)$ or $E$ if there is no risk of confusion. Each edge has its two end vertices associated to it.

\begin{remark}
We will assume throughout the paper that our graphs do not have vertices of degree $2$. Once we pass to quantum graphs two edges connected 
by a vertex of degree $2$ with Kirchhoff boundary condition behave exactly the same way as a single longer edge does.
\end{remark}

\begin{definition}
 A cycle in a graph is a closed walk that does not repeat any edges or vertices. Whenever we use the word cycle in this paper we mean it 
in this graph theoretical sense and not in a homological sense.
\end{definition}

\begin{definition}
\label{overlap}
Let $\gamma_1$ and $\gamma_2$ be two oriented cycles  in a graph. 
We say they have edges of positive overlap if they have an edge in common and pass through it in the same direction. We say they have edges of 
negative overlap if they have an edge in common and pass through it in opposite directions.
\end{definition}

Note that two cycles can have both edges with positive and negative overlap.

\begin{lemma}
\label{basis_of_cycles}
 Every graph admits a basis of its homology that consists of cycles. 
\end{lemma}
\begin{proof}
 Pick a spanning tree of the graph. Associate to each edge of $G$ not in the spanning tree the cycle that consists of this edge and the 
path in the spanning tree that connects its end points. This collection of cycles is a basis of the homology. 
\end{proof}

\begin{definition}
 We call a graph with no leaves, that is vertices of degree $1$, a leafless graph. 
\end{definition}

\subsection{The block structure of a graph}
%%%%%%%%%%%%%%%%%%%%%%%%%%%%%%%%%%%%%%%%%%%%%%%%%%%%%%%%%%%%%%%%%%%%%%%%%%%%%%%%%%%%%%%%%%%%%%%%%%%%%%%%%%%%%%%%%%%%%%%%%%%%%%%

The block structure provides a broad view of the structure of a graph.

\begin{definition}
\label{k-connected}
 A graph $G$ is called $k$-connected if $|V|>k$ and $G\setminus X$ is connected for any subset of vertices $X$ with $|X|<k$.

Here $G \setminus X$ is the subgraph of $G$ with vertex set $V\setminus X$ and the edges of $G$ between these vertices.
\end{definition}

\begin{definition}
 A vertex $v$ in $G$ is called a cut vertex if $G \setminus \{v\}$ is disconnected.
\end{definition}

\begin{definition}
Consider the set of all cycles in the graph. Declare two cycles equivalent if they share at least one edge. This generates an 
equivalence relation.  We define the set of blocks to be the set of equivalence classes.
\footnote{This is a slight deviation from the standard definition. It is changed to allow graphs with loops and multiple edges. 
Edges that are not part of any cycle are not part of any block in our 
definition. Usually these edges are counted as blocks, too. }
\end{definition}

\begin{definition}
\label{define_block}
We define the block structure of a graph as follows. Each block in the graph is replaced by a small circle that we call a fat vertex. 
The cut vertices contained in this block correspond to the different attaching points on the fat vertex. For loops we interpret 
their vertex as the cut vertex where they are attached to the rest of the graph. 

All other blocks or 
remaining edges sharing one of the cut vertices with the original block are connected at the respective attaching point on the fat vertex. 

It does not matter how the different attaching points are arranged around the fat vertex.
We explicitly allow several fat vertices to be directly connected to each other without an edge in between.
\footnote{Again this is a non standard definition. Our definition contains the same information about the graph as the standard 
one modulo the addition of loops and multiple edges.}
\end{definition}

\begin{example}
This is an example of a graph and its block structure. Note that for simplicity of recognition all blocks in the graphs are either loops or copies of the 
complete graph on $4$ vertices, $K_4$.
\begin{figure}[h]
\centering
\scalebox{0.4}{\includegraphics{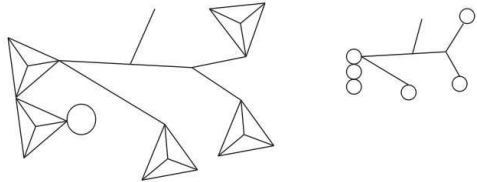}}
\caption{a quantum graph and its block structure}
\end{figure}
\end{example}

\begin{remark}
\label{cycle_single_block}
Any cycle in the graph is confined to a single block.
Thus the vertices and edges in the block structure never form a cycle and the block structure has a tree like shape.
\end{remark}

We will phrase the next lemma in the context of quantum graphs as we will need it later on. Combinatorial graphs can be seen as 
quantum graphs where all edges have length one.

\begin{lemma}
 \label{tree_from_leaves}
Let $G$ be a quantum tree graph with no vertices of degree $2$. Then the set of distances between any pair of leaves determines both the 
combinatorial tree graph underlying $G$ and all individual edge lengths.
\end{lemma}
\begin{proof}
Given three leaves $B_i$, $B_j$ and $B_k$ the restriction of the tree to the paths between these leaves is 
shaped like a star. We will denote the length of the three branches by $l_i, l_j$ and $l_k$. The distances between the leaves determine 
the quantities $l_i+l_j$, $l_i+l_k$ and $l_j+l_k$ and thus the three individual lengths $l_i, l_j$ and $l_k$. This means that 
given a path between two leaves $B_i$ and $B_j$ and a third leaf $B_k$ we can find both the point on the path from $B_i$ to $B_j$ 
where the paths from $B_i$ and $B_j$ to $B_k$ branch away and the length of the path from this point to $B_k$. 

We will use this fact repeatedly and proceed by induction on the number of leaves. 

If there are only two leaves the tree consists of a single interval with length the distance between the two leaves. 

Suppose we already have a quantum tree graph 
with leaves $B_1, \ldots, B_{n-1}$. We now want to attach a new leaf $B_n$. We will first look at the leaves $B_1$ and $B_2$ and
 find the point on the path from $B_1$ to $B_2$ where the paths to $B_n$ branch away. If this point is not a 
vertex of the tree, we create a new vertex and attach the leaf $B_n$ on an edge of suitable length $l_n$. If this point is a vertex of the 
tree we know that the attachment point of $B_n$ has to lie on the subtree branching away from the path from $B_1$ to $B_2$ starting at 
that vertex. Pick a leaf on this subtree, without loss of generality $B_3$ 
and look at the path from $B_1$ to $B_3$. We can again find the point on that path where the paths to $B_n$ branch away. If this point 
is not a vertex of the tree we found the attachment point, otherwise we have reduced our search to an even 
smaller subtree. Continuing this process we will eventually end up with an attachment point on an edge or on a subtree that consists 
of a single vertex. 
\end{proof}

\subsection{Planarity of graphs}
%%%%%%%%%%%%%%%%%%%%%%%%%%%%%%%%%%%%%%%%%%%%%%%%%%%%%%%%%%%%%%%%%%%%%%%%%%%%%%%%%%%%%%%%%%%%%%%%%%%%%%%%%%%%%%%%%%%%%%%%%%%%%%%%%

The edge space of a graph is the $F_2$-vector space over the set of (unoriented) edges of the graph. The cycle space is the subspace 
generated by cycles in the graph.

Given an embedding into $\R^2$ of a planar graph the faces of the embedding are the disconnected components of $\R^2 \setminus G$.

\begin{theorem}
MacLane (1937) \cite{Diestel05}\\
A graph is planar if and only if its cycle space has a simple basis.
\end{theorem}

Simple means that each edge is part of at most 2 cycles in the basis.

\begin{corollary} 
\label{simple=planar}
A graph is planar if and only if it admits a basis of its homology consisting of oriented cycles having no positive overlap.
\end{corollary}
\begin{proof}
Each cycle is confined to a single block of the graph and two cycles in different blocks share at most a single vertex and thus have zero 
overlap. Thus it is sufficient to prove the statement for $2$-connected graphs.

Assume $G$ is planar and $2$-connected and choose an embedding into $\R^2$. The set of boundaries of faces with the exception of the outer 
face forms a 
basis of the homology that consists of cycles and is simple, see \cite{Diestel05}. We orient all basis cycles counterclockwise. Then no two 
of them can run through the same edge in the same direction as no basis cycle can lie inside another basis cycle. Thus there are 
no edges of positive overlap.

Any basis of the homology where every basis element can be represented by a cycle in the graph gives rise to a basis of the cycle 
space consisting of exactly these cycles.
Thus if the graph is not planar any 
basis of cycles of the homology is not simple by MacLanes theorem. 
Therefore there exists an edge that is part of three basis cycles. No matter how we orient these three cycles, two of them have to go 
through this edge with the same orientation and thus have edges of positive overlap.
\end{proof}

\begin{definition}
\label{non-positive_basis}
 We call a basis without edges of positive overlap a non-positive basis of the graph and remark that a non-positive basis is always simple.
\end{definition}

If $G$ is $2$-connected and planar we can find a simple basis by picking the boundaries of faces. This proposition states that the converse is 
true, too.

\begin{proposition}
\cite{MoharThomassen01}
\label{positive=facial}
 Given a simple basis of the cycle space of a $2$-connected planar graph there exists an embedding into $\R^2$ such that all 
basis elements are boundaries of faces.
\end{proposition}

\subsection{Dual graphs}
%%%%%%%%%%%%%%%%%%%%%%%%%%%%%%%%%%%%%%%%%%%%%%%%%%%%%%%%%%%%%%%%%%%%%%%%%%%%%%%%%%%%%%%%%%%%%%%%%%%%%%%%%%%%%%%%%%%%%%%%%%%%%%%

Planar graphs have a notion of a dual graph. We will present two different ways of defining it and list some properties.

\begin{definition}
Given a planar embedding of a graph $G$ we define the geometric dual graph $G^*$ associated to this embedding. 
The vertices of $G^*$ are the 
faces in the embedding of $G$. The number of edges joining to vertices in $G^*$ is the number of edges 
that the corresponding faces in $G$ have in common. 
\end{definition}

\begin{definition}
 A cut of a graph $G$ is a subset of (open) edges $S$ such that $G\setminus S$ is disconnected. A cut is minimal if no proper subset of $S$ is 
a cut.
\end{definition}

\begin{definition}
 Given a planar graph $G$ a graph $G^*$ is an abstract dual of $G$ if there is a bijective map $\psi : E(G) \ra E(G^*)$ such that 
for any $S \subseteq E(G)$ the set $S$ is a cycle in $G$ if and only if $\psi(S)$ is a minimal cut in $G^*$.
\end{definition}

\begin{proposition}
\cite{Diestel05}
A planar graph can have multiple non isomorphic abstract duals. Any geometric dual of a planar graph is an abstract dual and vice versa. The 
dual of a 
planar graph is planar and $G$ is an abstract dual of $G^*$. If $G$ is $3$-connected than $G^*$ is unique up to isomorphism.
\end{proposition}

\begin{definition}
 We call two graphs $G$ and $H$ $2$-isomorphic if there is a bijection between their edge sets that carries cycles to cycles. Note 
that this does not imply that the graphs are isomorphic.
\end{definition}

\begin{lemma}
\label{2-isomorphic}
Two planar graphs $G$ and $H$ are $2$-isomorphic if and only if they have the same set of abstract duals.
\end{lemma}
\begin{proof}
 Let $\varphi : E(G) \ra E(H)$ be a $2$-isomorphism and let $G^*$ be an abstract dual of $G$ with edge bijection $\psi$. Then 
$\psi \circ \varphi^{-1}$ is an edge bijection that makes $G^*$ an abstract dual of $H$.

Let $G$ and $H$ have the same abstract duals and let $G^*$ be an abstract dual. Let $\psi_1$ be an edge bijection between $G$ and $G^*$ and 
let $\psi_2$ be an edge bijection between $H$ and $G^*$. Then $\psi_2^{-1}\circ\psi_1$ is a $2$-isomorphism between $G$ and $H$.
\end{proof}

\section{Differential forms}
%%%%%%%%%%%%%%%%%%%%%%%%%%%%%%%%%%%%%%%%%%%%%%%%%%%%%%%%%%%%%%%%%%%%%%%%%%%%%%%%%%%%%%%%%%%%%%%%%%%%%%%%%%%%%%%%%%%%%%%%%%%%%%%%

The concept of differential forms on a quantum graph was introduced in \cite{GaveauOkada91}. 

Let $G$ be a quantum graph, let $E$ and $V$ be the set of edges and vertices. Let $L_e$ denote the length of the edge $e$.
Let $\{ e \sim v \}$  denote the set of edges $e$ adjacent to a vertex $v$.

\begin{definition}
 A vector field $X$ on $G$ is a smooth vector field on each edge, seen as a closed interval. In particular a vector field is multivalued 
at the vertices. 
\end{definition}

Let $\nu_{v,e}$ denote the outward unit normal for the edge $e$ at the vertex $v$ .
Let $X_1$ be an auxiliary vector field that is real and has constant length $1$ on all edges.

\begin{definition}
A $0$-form $f$ on $G$ is a function that is $C^{\infty}$ on the edges, that is continuous and that satifies the Kirchhoff boundary 
condition 
\begin{equation*}
\sum_{e \sim v }\nu_{v,e}(f|_e)=0 
\end{equation*}
at all vertices $v \in V$. We denote the space of $0$-forms by $\Lambda^0$.
\end{definition}

\begin{definition}
A $1$-form $\alpha$ on $G$ consists of a $1$-form $\alpha_e$ on each closed edge $e$ such that $\alpha$ satisfies the boundary condition
\begin{equation*}
 \sum_{e \sim v} \alpha_e(\nu_{v,e})=0
\end{equation*}
at all vertices $v \in V$. We denote the space of $1$-forms by $\Lambda^1$.
\end{definition}

\begin{definition}
 For a real $1$-form $\alpha$ we define an operator $d_{\alpha} : \Lambda^0 \ra \Lambda^1$ through the requirement 
\begin{equation*}
(d_{\alpha}f)(X):=X(f)+2\pi i\alpha(X)f
\end{equation*}
for all vector fields $X$. We denote the operator $d_0$ by $d$.
\end{definition}

\begin{definition}
 We define a hermitian inner product on $\Lambda^0$ by 
\begin{equation*}
 (f,g) :=\int_G f(x)\overline{g(x)}dx
\end{equation*}
and denote the completion of $\Lambda^0$ with respect to this inner product by $K^0$.
\end{definition}

\begin{definition}
We define a hermitian inner product on $\Lambda^1$ by 
\begin{equation*}
 (\alpha,\beta) :=\int_{G}\alpha(X_1)\overline{\beta(X_1)}dx=\sum_{e\in E}\int_0^{L_e}\alpha_e(X_1|_e)\overline{\beta_e(X_1|_e)}dx
\end{equation*}
This is clearly independent of the choice of the auxiliary vector field $X_1$.
\end{definition}

We are now going to define the adjoint of $d_{\alpha}$. Formally it should satisfy
\begin{equation*}
 (d_{\alpha}^* \beta, f)=(\beta, d_{\alpha}f)
\end{equation*}
for all $f \in \Lambda^0$. We have\
\begin{eqnarray*}
 (\beta, d_{\alpha}f) &=& \int_G\beta(X_1)\overline{X_1(f)}dx - 2\pi i\int_G\beta(X_1)\overline{\alpha(X_1)f}dx \\
&=& -\int_G X_1(\beta(X_1))\overline{f}dx+\sum_{v\in V } \overline{f(v)}\sum_{e \sim v} \beta_e(\nu_{v,e})\\
&& -2\pi i\int_G\alpha(X_1)\beta(X_1)\overline{f}dx 
\end{eqnarray*}
where we used integration by parts. The sum term vanishes because of the boundary condition on $1$-forms.
So we find that $d^*_{\alpha}$ satisfies
\begin{equation*}
 d_{\alpha}^* \beta = - X_1(\beta(X_1))-2\pi i\alpha(X_1) \beta(X_1) = d^*\beta - 2\pi i \alpha(X_1)\beta(X_1)
\end{equation*}
which again is independent of the choice of $X_1$.

\begin{definition}
 For each edge $e \in E$ we define the Sobolov space $W_2(e)$ as the closure of $C^2([0,L_e])$ with respect to the norm 
$||f||_2^2:=\sum\limits_{j=0}^2 \int\limits_0^{L_e} |f^{(j)}(x)|^2dx$.
 \end{definition}

\begin{definition}
 We define a Schr\"odinger type operator 
\begin{equation*}
 \Delta_{\alpha}:=d_{\alpha}^* d_{\alpha}
\end{equation*}
 on $\Lambda^0$. We extend its domain to
\begin{eqnarray*}
 Dom(\Delta_{\alpha}):=\left\{ f \in K^0 \middle| \forall e \in E : f|_e \in W_2(e), \forall v\in V : \sum_{e \sim v}\nu_{v,e}(f|_e)=0   \right\}
\end{eqnarray*}
\end{definition}

\begin{theorem}
 \cite{GaveauOkada91}
We have $H^1(G,\C)=\Lambda^1 \slash d(\Lambda^0)$. 
Thus the definitions of $1$-forms and $0$-forms produce the expected deRham cohomology.
\end{theorem}

\begin{proposition}
\label{1form_invariance}
 Let  $\alpha \in \Lambda^1$ and $\psi \in \Lambda^0$ be real and let $\beta=\alpha+d\psi$. Let $f$ be an 
eigenfunction of $\Delta_{\alpha}$ with eigenvalue $\lambda$. Then $e^{-2\pi i\psi}f$ is an eigenfunction of $\Delta_{\beta}$ with the same 
eigenvalue. That is two forms that differ by an exact $1$-form have the same spectrum.
\end{proposition}
\begin{proof}
We have
\begin{eqnarray*}
&& d_{\beta}^*d_{\beta}\left(e^{-2\pi i\psi}f\right)\\
&=& d_{\beta}^*\left( d(e^{-2\pi i\psi}f)+2\pi i e^{-2\pi i\psi}f\alpha + 2\pi i e^{-2\pi i\psi}fd\psi \right)\\
&=& d_{\beta}^*\left( e^{-2\pi i\psi}df+2\pi i e^{-2\pi i\psi}f\alpha \right)\\
&=& d_{\beta}^* \left(e^{-2\pi i\psi}d_{\alpha}f\right) \\
&=& d^*\left(e^{-2\pi i\psi}d_{\alpha}f\right) - 2\pi i\alpha(X_1) e^{-2\pi i\psi}d_{\alpha}f(X_1) - 2\pi i d\psi(X_1) e^{-2\pi i\psi}d_{\alpha}f(X_1)\\
&=& e^{-2\pi i\psi}d^*d_{\alpha}f -2\pi i\alpha(X_1) e^{-2\pi i\psi}d_{\alpha}f(X_1) \\
&=& e^{-2\pi i\psi}d_{\alpha}^*d_{\alpha}f
\end{eqnarray*}
Thus $f$ is an eigenfunction for $\Delta_{\alpha}$ if and only if $e^{-2\pi i\psi}f$ is an eigenfunction for $\Delta_{\beta}$ with the same eigenvalue.
\end{proof}

\begin{remark}
 Note that $Spec_{\alpha}(G)$ depends only on the coset of $[\alpha]$ in $H^1_{dR}(G,\R) \slash H^1_{dR}(G,\Z)$. 
\end{remark}

\begin{definition}
Let $Spec_{\alpha}(G)$ be the spectrum of the operator $\Delta_{\alpha}$. 
 We define the Bloch spectrum $Spec_{Bl}(G)$ of a quantum graph to be the map that associates to each $[\alpha]$ the 
spectrum $Spec_{\alpha}(G)$ where $[\alpha] \in H^1_{dR}(G,\R) \slash H^1_{dR}(G,\Z)$.

Note that we assume that we only know $H^1_{dR}(G,\R) \slash H^1_{dR}(G,\Z)$ as an abstract torus without any Riemannian structure.
\end{definition}

\begin{definition}
 We say that two quantum graphs $G$ and $G'$ are Bloch isospectral if there is a Lie group isomorphism  
$\Phi: H^1_{dR}(G,\R)/H^1_{dR}(G,\Z)  \ra H^1_{dR}(G',\R)/ H^1_{dR}(G',\Z)$ 
such that $Spec_{\alpha}(G)=Spec_{\Phi(\alpha)}(G')$ for all $[\alpha] \in H^1_{dR}(G,\R)/H^1_{dR}(G,\Z)$.
\end{definition}

Note that  if $G$ is a tree graph the entire Bloch spectrum just consists of the spectrum of the Laplacian $\Delta_0$
and thus does not contain any additional information.

\section{The Bloch spectrum via characters}
%%%%%%%%%%%%%%%%%%%%%%%%%%%%%%%%%%%%%%%%%%%%%%%%%%%%%%%%%%%%%%%%%%%%%%%%%%%%%%%%%%%%%%%%%%%%%%%%%%%%%%%%%%%%%%%%%%%%%

In this chapter we will introduce the Bloch spectrum using characters of the first fundamental group and show that the two notions 
are equivalent.

Let $\tilde{G}$ be the universal cover of $G$ and let $\pi_1(G)$ denote the fundamental group. Then $\pi_1(G)$ acts by deck transformations 
 on $\tilde{G}$. Let $\chi: \pi_1(G) \ra \C^*$ be a character of $\pi_1(G)$.

We will study functions $\tilde{f}: \tilde{G} \ra \C$ that are continuous and satisfy Kirchhoff boundary conditions at the vertices 
and that obey the transformation law
\begin{equation*}
\tilde{f}(\gamma x)=\chi(\gamma) \tilde{f}(x) 
\end{equation*}
for all $x\in \tilde{G}$ and $\gamma \in \pi_1(G)$. We refer to the space of these functions as $\Lambda^0_{\chi}(\tilde{G})$.

We associate to the character $\chi$ the spectrum of the standard Laplacian $d^*d$ on $\tilde{G}$ restricted to functions in 
$\Lambda^0_{\chi}(\tilde{G})$, we will denote it by $Spec(G,\chi)$.

\begin{definition}
We call the map that associates to each character $\chi$ of $\pi_1(G)$ the spectrum $Spec(G, \chi)$ the $\pi_1$-spectrum of $G$.
\end{definition}

\begin{theorem}
The Bloch spectrum $Spec_{Bl}(G)$ and the $\pi_1$-spectrum of a graph are equal.
There is a one-to-one correspondence $[\alpha] \mapsto \chi_{\alpha}$ between $H^1_{dR}(G,\R) \slash H^1_{dR}(G,\Z)$ and the set of characters 
of $\pi_1(G)$. It is given by 
\begin{equation*}
 \chi_{\alpha}(\gamma)=e^{-2\pi i\int_{\gamma}\alpha}
\end{equation*}
It induces the equality  $Spec(G, \chi_{\alpha})=Spec_{\alpha}(G)$.
\end{theorem}
\begin{proof}
The integral does not depend on either the representative in $\pi_1(G)$ nor on the one in $H^1_{dR}(G,\R)$ so this gives a well defined map. 
We also have $\chi_{\alpha}(\gamma_1 \cdot \gamma_2)=\chi_{\alpha}(\gamma_1)\chi_{\alpha}(\gamma_2)$ so this defines a character.

Let $f: G \ra \C$ and let $\tilde{f} :\tilde{G} \ra \C$ be the lift of $f$. Let $\tilde{\alpha}$ 
be the pullback of $\alpha$. 
As $H^1(\tilde{G})$ is trivial $\tilde{\alpha}$ is exact and there exists a function $\tilde{\varphi}: \tilde{G} \ra \C$ such that 
$\tilde{\alpha}=d\tilde{\varphi}$.

Let $\tilde{g}(x):=e^{-2\pi i\tilde{\varphi}(x)}\tilde{f}(x)$. We claim that $\tilde{g}$ is an eigenfunction in the $\pi_1$-spectrum if and only 
if $\Delta_{\alpha} f=\lambda f$. We need to show that $\tilde{g} \in \Lambda^0_{\chi_{\alpha}}(\tilde{G})$ and that 
$\Delta \tilde{g}= \lambda \tilde{g}$.

Let $\gamma \in \pi_1(G)$ and let $\tilde{\gamma}$ be the (unique) path in $\tilde{G}$ from $x$ to $\gamma x$. We have 
$\tilde{\varphi}(\gamma x)-\tilde{\varphi}(x)=\int_{\tilde{\gamma}}d\tilde{\varphi}$ 
by Stokes theorem. So we get 
\begin{equation*}
 \tilde{g}(\gamma x)=e^{-2\pi i\tilde{\varphi}(\gamma x)}\tilde{f}(\gamma x)=
e^{-2\pi i\int_{\tilde{\gamma}}\tilde{\alpha}}e^{-2\pi i\tilde{\varphi}(x)}\tilde{f}(x)= \chi_{\alpha}(\gamma)\tilde{g}(x)
\end{equation*}
By proposition \ref{1form_invariance} we have\\
\begin{equation*}
\Delta \tilde{g} = \Delta e^{-2\pi i\tilde{\varphi}}\tilde{f} = e^{-2\pi i\tilde{\varphi}}\Delta_{\tilde{\alpha}} \tilde{f}
\end{equation*}
Thus $\tilde{g}$ is an eigenfunction with eigenvalue $\lambda$ if and only if $f$ is.
\end{proof}

\begin{remark}
This theorem mirrors a similar result for tori, see \cite{Guillemin90}. 
\end{remark}

\section{The Albanese torus}
%%%%%%%%%%%%%%%%%%%%%%%%%%%%%%%%%%%%%%%%%%%%%%%%%%%%%%%%%%%%%%%%%%%%%%%%%%%%%%%%%%%%%%%%%%%%%%%%%%%%%%%%%%%%%

\begin{definition}
We call a $1$-form $\alpha$ harmonic if $d^*\alpha \in \Lambda^0$ and $dd^*\alpha = 0$. 
\end{definition}

\begin{lemma}
\cite{GaveauOkada91}
A $1$-form $\alpha$ is harmonic if and only if $\alpha(X_1)$ is constant on all edges where $X_1$ is the auxiliary vector field of length one.
\end{lemma}

\begin{lemma}
\cite{GaveauOkada91} 
Any $\beta \in \Lambda^1$ admits a unique Hodge decomposition of the form
\begin{equation*}
 \beta=d\psi+\tilde{\beta}
\end{equation*}
where $\psi \in \Lambda^0$ and $\tilde{\beta}$ is harmonic.

Thus each cohomology class has exactly one harmonic representative.

If $\beta$ is real then so are $\psi$ and $\tilde{\beta}$.
\end{lemma}

\begin{definition}
 We define an inner product on $H^1(G,\R)$ by 
\begin{equation*}
 ([\alpha],[\beta]):=(\tilde{\alpha}, \tilde{\beta})
\end{equation*}
where $\tilde{\alpha}$ and $\tilde{\beta}$ are the unique harmonic representatives of $[\alpha]$ and $[\beta]$.
\end{definition}

Kotani and Sunada defined the notion of the Albanese torus of a combinatorial graph in \cite{KotaniSunada00}. We will generalize this to 
quantum graphs. If the quantum graph is equilateral our definition recovers theirs.

Let $or(E)$ be the set of oriented edges, we call an element $b\in or(E)$ a bond. Let $\overline{b}$ denote a reversal of orientation. Let $o(b)$ and $t(b)$ be the 
origin and terminal vertex of a bond $b$.

Let $A$ be an abelian group, the 
coefficients of our homology. Let $C_0(G,A)$ be the free $A$-module with generators in $V$. Let $C_1(G,A)$ be the $A$-module generated by 
$or(E)$ modulo the relation $\overline{b}=-b$.
The boundary map $\partial: C_1(G,A) \ra C_0(G,A)$ is defined by $\partial (b):=t(b)-o(b)$ and linearity. We then have $H_1(G,A)=ker(\partial)$.

We have the natural pairing $([\alpha], [p]) \mapsto \int_p \alpha$ for any $[\alpha] \in H^1(G,\R)$ and $[p] \in H_1(G,\R)$. 
This makes these two spaces dual to each other and induces an inner product on $H_1(G,\R)$.

\begin{lemma}
\label{inner_product}
 This inner product is equivalent to the one we get on $H_1(G,\R)$ as a subspace of $C_1(G,\R)$ with the inner 
product given by 
\begin{equation*}
 \label{cases}
e\cdot e' = \cases{l(e) & $e=e'$\\
		 -l(e) &$ e=\overline{e'}$\\
		0 & $ $otherwise}
\end{equation*}
on edges and bilinear extension.
\end{lemma}

This might seem an awkward inner product if one thinks of vectors but the better analogy would be to think of characteristic functions of 
sets in $\R^n$ with an $L^2$ inner product.

\begin{remark}
The inner product plays well with our notion of edges of positive and negative overlap in definition \ref{overlap}. The inner product 
of two cycles is equal to the difference between the length of the edges of positive and negative overlap. 
\end{remark}

\begin{definition}
 The Albanese torus of a quantum graph is given by
\begin{equation*}
 Alb(G)= H_1(G,\R) \slash H_1(G,\Z)
\end{equation*}
The Jacobian torus of a quantum graph is given by
\begin{equation*}
 Jac(G)= H^1(G,\R) \slash H^1(G,\Z)
\end{equation*}
The inner products induce the structure of Riemannian manifolds and make them into dual tori.
\end{definition}

\section{A trace formula}
%%%%%%%%%%%%%%%%%%%%%%%%%%%%%%%%%%%%%%%%%%%%%%%%%%%%%%%%%%%%%%%%%%%%%%%%%%%%%%%%%%%%%%%%%%%%%%%%%%%%%%%%%%%%%%%%%%%%%%%%%%%%%%%%%%%%%%%%%%%

The spectrum of a single Schr\"odinger type operator determines a trace formula. There are various different versions.
All of them contain essentially the same information about the quantum graph. We will 
use the following.

\begin{theorem}
\cite{KottosSmilansky99}
\label{trace_formula}
 The spectrum $Spec_{\alpha}(G)=\{\lambda_n\}_n$ of the operator $\Delta_{\alpha}$ determines the following exact wave trace formula.
\begin{eqnarray*}
\sum_n \delta(\lambda-\lambda_n)
=\frac{\mathcal{L}}{\pi}+\chi(G)\delta(\lambda)+
\frac{1}{2\pi} \sum_{p\in PO}\left(\mathcal{A}_p(\alpha)e^{i\lambda l_p}+\overline{\mathcal{A}_p}(\alpha)e^{-i\lambda l_p}\right)
\end{eqnarray*}
\end{theorem}
Here the first sum is over the eigenvalues including multiplicities, the $\delta$ are Dirac $\delta$ distributions.

$\mathcal{L}$ denotes the total length of the graph. $\chi(G)$ denotes the Euler characteristic of the graph.

The second sum is over all periodic orbits, $l_p$ denotes the length of a periodic orbit. A periodic orbit is an oriented 
closed walk in the graph.

The coefficients $\mathcal{A}_p(\alpha)$ are given by
\begin{equation*}
 \mathcal{A}_p(\alpha)=\tilde{l}_pe^{2\pi i\int_p\alpha}\prod_{b\in p} \sigma_{t(b)}
\end{equation*}
Here $\tilde{l}_p$ is the length of the primitive periodic orbit that $p$ is a repitition of. The $e^{2\pi i\int_p\alpha}$ is the phase factor or 
`magnetic flux'. The product is over the sequence of oriented edges or bonds in the periodic orbit. 
The vertex scattering coefficients $\sigma_{t(b)}$ at the terminal vertex $t(b)$ of each bond is given by
$\sigma_{t(b)}=-\delta_{t(b)}+\frac{2}{\deg(t(b))}$. Here $\delta_{t(b)}$ is defined to be equal to one if the periodic orbit is 
backtracking at the vertex $t(b)$ and zero otherwise.

\begin{remark}
\label{contractible}
The phase factor $e^{2\pi i\int_p\alpha}$ of a periodic orbit only depends on its homology class by Stokes theorem. For a 
contractible periodic orbit it is equal to $1$. 
\end{remark}

\begin{corollary}
\label{fourier_trace_formula}
\cite{KottosSmilansky99}
The Fourier transform of this trace formula is given by:
\begin{equation*}
\sum_n e^{-il\lambda_n}=2\mathcal{L}\delta(l)+\chi(G)+ 
\sum_{p\in PO}\mathcal{A}_p(\alpha)\delta(l-l_p)+\overline{\mathcal{A}}_p(\alpha)\delta(l+l_p)
\end{equation*} 
\end{corollary}

\section{The homology of a quantum graph}
%%%%%%%%%%%%%%%%%%%%%%%%%%%%%%%%%%%%%%%%%%%%%%%%%%%%%%%%%%%%%%%%%%%%%%%%%%%%%%%%%%%%%%%%%%%%%%%%%%%%%%%%%%%%%%%%%%%%%%%%%%%%%%%%%%%%%%%%%%

In this chapter we will analyze the spectrum and the trace formula and extract information about the homology of the graph from it.

Before we state and prove the main theorem of this chapter we need a few definitions and a technical lemma.

\begin{definition}
We call a periodic orbit minimal if it has minimal length within its homology class.
\end{definition}

\begin{remark}
Note that in general a given element in the homology might have more than one minimal periodic orbit that represents it.

On the other hand all closed walks that contain no edge repititions and in particular all cycles are minimal. Cycles are also the unique 
minimal periodic orbit in their homology class.
\end{remark}

\begin{definition}
\label{generic}
We call an $\alpha$ generic if 
the image of the ray  $t\alpha$ in the torus $H^1_{dR}(G,\R) \slash H^1_{dR}(G,\Z)$ is dense. The $\alpha$'s with this property are dense. We pick and fix a single generic 
$\alpha$. 
\end{definition}

\begin{definition}
\label{frequencies}
To the fixed generic $\alpha$ we associate the following data.
\begin{enumerate}
 \item Let $\Psi$ be the linear map $ \Psi: H_1(G,\Z) \ra \R$ given by 
$[p] \mapsto 2\pi \int_p \alpha $. It associates to each periodic orbit its magnetic flux. 
\item We call the absolute values of the magnetic fluxes $\mu:=|\Psi([p])|=2\pi \left|\int_p \alpha\right|$ the 
frequencies associated to $\alpha$.
\item We will denote the length of the minimal periodic orbit(s) associated to a frequency $\mu$ by $l(\mu)$.
\end{enumerate}
\end{definition}

\begin{remark}
The map $\Psi$ is two-to-one (except at zero) because we picked $\alpha$ to be generic.
The set of all frequencies $\mu$ union their negatives $-\mu$ and zero
forms a finitely generated free abelian subgroup of $\R$ that 
is isomorphic to  $H_1(G, \Z)$ via the map $\Psi$. 
\end{remark}

\begin{lemma}
\label{determine_cosine}
 Let $f$ be a function that is a linear combination of several cosine waves with different (positive) frequencies. 
\begin{equation*}
 f(t)=\sum_{j=1}^k\nu_j\cos(\mu_j t)
\end{equation*}
Then the values $f(t)$ for $t \in [0,\varepsilon)$ determine both $k$ and the individual frequencies $\mu_1, \ldots, \mu_k$.
\end{lemma}
\begin{proof}
 Assume without loss of generality that $0 < \mu_1 < \ldots < \mu_k$. We will show that we can determine $\mu_k$ and $\nu_k$ and then use 
induction.
We will look at the collection of derivatives of $f$ at $t=0$. We have
\begin{equation*}
 f^{(2n)}(0)=(-1)^n\sum_{j=1}^k\nu_j \mu_j^{2n}
\end{equation*}
There exists a unique number $\lambda >0$ such that\\
\begin{equation*}
 -\infty < \lim_{n \ra \infty} \frac{f^{(2n)}(0)}{(-\lambda)^n} < \infty \hspace{1cm} and \hspace{1cm} 
\lim_{n \ra \infty} \frac{f^{(2n)}(0)}{(-\lambda)^n} \neq 0
\end{equation*}
and we have $\lambda=\mu_k^2$ and $\lim_{n \ra \infty} \frac{f^{(2n)}(0)}{(-\lambda)^n}=\nu_k$. We can now look at the new function, 
\begin{equation*}
 \tilde{f}(t):=f(t)-\nu_k\cos(\mu_k t)
\end{equation*}
repeat the process and determine $\mu_{k-1}$ and $\nu_{k-1}$. After finitely many steps we will end up with the constant function $0$.
\end{proof}

The following theorem is the key link between the Bloch spectrum and the quantum graph. All other theorems are just consequences 
of this one.

\begin{theorem}
\label{Bloch_H1}
 Given a generic $\alpha$, see definition \ref{generic}, the part of the Bloch spectrum $Spec_{t\alpha}(G)$ for $t\in [0,\varepsilon)$ 
 determines the length of the minimal periodic orbit(s) of each element in $H_1(G, \Z)$. 
\end{theorem}

\begin{proof} 
We will show we can read off the set of frequencies $\mu$, see definition \ref{frequencies}, associated to the generic $\alpha$ 
from the Bloch spectrum and determine the length $l(\mu)$ for each frequency.

We will look at the continuous family of $1$-forms $\alpha(t)=t\alpha$ and the associated 
operators $\Delta_{\alpha(t)}$ for our fixed generic $\alpha$. If we plug the eigenvalues of these operators 
into the Fourier transform of the trace formula we get a family of distributions. Each of these distributions is a locally 
finite sum of Dirac-$\delta$-distributions (plus a constant term). The support of each of these $\delta$-distributions is the length of 
the periodic orbit it is associated to and thus depends only on the underlying graph and not on the $1$-form, see  
\ref{fourier_trace_formula}. 

Any periodic orbit $p$ that is homologically non-trivial has a corresponding partner which is the 
same closed walk with opposite orientation. Their coefficients are related by 
$\mathcal{A}_p(\alpha)=\overline{\mathcal{A}_{\overline{p}}}(\alpha)$ 
as the connectivity part and the length are the same and the magnetic flux changes sign. Thus for each such pair we would observe a 
factor of the form $2{\rm Re} \mathcal{A}_p(t\alpha)=:2{\rm Re} \mathcal{A}_p(t)$ in the Fourier transform of the trace formulae for $\Delta_{\alpha(t)}$. We have 
\begin{equation*}
 2{\rm Re} \mathcal{A}_p(t)=2{\rm Re} \left( \tilde{l}_pe^{2\pi i t\int_p\alpha}\prod_{e\in p} \sigma_{t(e)}\right) =\nu \cos\left(2\pi \int_p\alpha t\right)
\end{equation*}
by theorem \ref{trace_formula} where $\nu=2\tilde{l}_p\prod_{e\in p} \sigma_{t(e)}$. So for each such pair of periodic orbits 
there is a magnetic flux $\Psi([p])=2\pi \int_p\alpha$ 
 and a factor $\nu$ that is always non zero. Moreover the factor $\nu$  is positive if the periodic orbit 
contains no backtracks. As the magnetic flux appears in a cosine wave we can only know its absolute value, that is the frequency, see definition 
\ref{frequencies}. 

Pick a length of periodic orbits $l$. 
If we look at the family of $1$-forms $\alpha(t)$ we get a continuous family of coefficients $\mathcal{A}^l(t)=\sum_{l_p=l}\mathcal{A}_p(t)$.
As we did not make any assumptions on the underlying quantum graph there can be 
multiple periodic orbits with the same length.
Thus each coefficient $\mathcal{A}^l(t)$ is a linear combination 
of a constant term and several cosine waves with 
different frequencies. The constant part comes from homologically trivial periodic orbits of length $l$. The cosine waves correspond to the 
homologically non-trivial periodic orbits of length $l$. We can now apply lemma \ref{determine_cosine} to the function  $\mathcal{A}^l(t)$ 
and read off all the frequencies occuring at that length.

As we go through the different lengths in the spectra starting at zero we will pick up a collection of different frequencies. Each 
frequency will appear multiple times at different lengths since there are multiple periodic orbits that represent the same element in the homology and 
thus have the same frequency.

The frequency corresponding to a particular element in $H_1(G, \Z)$ can only be realized by periodic orbits that 
represent this element in the homology. Going through the lengths starting at zero this frequency can appear at the earliest at the length of 
the corresponding minimal periodic orbit(s). The minimal periodic orbits need not be unique but as they are minimal they contain no backtracks. 
Thus their $\nu$ coefficients are all strictly bigger than $0$ so their sum cannot vanish and the frequency will indeed appear in the coefficient 
$\mathcal{A}^l(t)$ at the minimal length. 
This gives us the length $l(\mu)$ associated to each frequency $\mu$. 
\end{proof}

\begin{remark}
As we picked $\alpha$ to be generic the number of rationally independent frequencies is equal to $\dim H_1(G,\Z)$. 
Thus we can observe from the number of rationally independent frequencies whether an arbitrary $\alpha$ is generic or not.
\end{remark}

\begin{remark}
Without any genericity assumptions on the edge lengths in the quantum graph it can happen that there are multiple non-minimal periodic orbits that are homologous and of 
the same length. We would not be able to distinguish them directly in the trace formula, it can even happen that their $\nu$-coefficients cancel out 
and we would not observe them at all. 
\end{remark}

\begin{lemma}
\label{recognize_cycle}
Given a frequency $\mu$ the following two statements are equivalent:
\begin{enumerate}
\item The minimal periodic orbit associated to $\mu$ is a cycle in the graph.
 \item  There are no two frequencies $\kappa$, $\kappa'$, $\mu= |\kappa \pm \kappa'|$ with the property that 
$l(\kappa)+l(\kappa') \le l(\mu)$.
\end{enumerate}
\end{lemma}
\begin{proof}
 We will prove both directions by contradiction. 

Assume the minimal periodic orbit associated to $\mu$ is not a cycle, then it has to go through some vertex at 
least twice. Thus we can seperate the periodic orbit into two shorter periodic orbits. Let $\kappa$ and $\kappa'$ be the  frequencies associated to the two pieces. 
Then $\mu=|\kappa \pm \kappa'|$ and because the two pieces are not necessarily minimal $l(\kappa)+l(\kappa') \le l(\mu)$.

Conversly, suppose $\mu$ admits a decomposition $\mu=|\kappa \pm \kappa'|$. Let $c_{\mu}$, $c_{\kappa}$ and $c_{\kappa'}$ 
denote the minimal periodic orbits associated to the frequencies.  
If we have $l(\kappa)+l(\kappa') = l(\mu)$ then $c_{\mu}=c_{\kappa}\cup c_{\kappa'}$ and $c_{\kappa}$ and $c_{\kappa'}$ must have a vertex in common so $c_{\mu}$ is not a 
cycle. If we have $l(\kappa)+l(\kappa') < l(\mu)$ then
 the periodic orbits $c_{\kappa}$ and $c_{\kappa'}$ are disjoint and $c_{\mu}$ realizes the connection between them so it uses the edges 
between them twice and is not a cycle.
\end{proof}

\begin{remark}
\label{check_overlap_sign}
 Let $\mu_1$ and $\mu_2$ be two frequencies such that the associated minimal periodic orbits $c_1$ and $c_2$ are cycles. 
These cycles have an orientation induced from the $1$-form $\alpha$. The frequency $\mu_1+\mu_2$ corresponds to the pair of periodic orbits 
that is homologous to $c_1\cup c_2$ and $-(c_1\cup c_2)$. Thus if $l(\mu_1+\mu_2) < l(\mu_1)+l(\mu_2)$ 
then $c_1$ and $c_2$ have edges of negative overlap. The frequency $|\mu_1-\mu_2|$ corresponds to the pair of periodic orbits 
that is homologous to $c_1\cup (-c_2)$ and $(-c_1)\cup c_2$. Thus if $l(\mu_1-\mu_2) < l(\mu_1)+l(\mu_2)$ 
then $c_1$ and $c_2$ have edges of positive overlap. 
\end{remark}

\begin{theorem}
\label{determine_Albanese}
The Bloch spectrum of $G$ determines the Albanese torus $Alb(G)$ as a Riemannian manifold.
\end{theorem}
\begin{proof}
Pick a minimal set of generators $\mu_1, \ldots, \mu_n$ of the group spanned by the frequencies such that the associated minimal periodic orbits 
are all cycles. Such a basis exists by lemma \ref{basis_of_cycles}. 
Associate to them a set of 
vectors $v_1, \ldots, v_n$ satisfying $|v_i|^2:=l(\mu_i)$ and $2\langle v_i,v_j \rangle := l(\mu_i+\mu_j)-l(|\mu_i-\mu_j|)$ for all $i \neq j$. 
This uniquely determines a torus with spanning vectors $v_1, \ldots, v_n$. 

If the cycles associated to $\mu_i$ and $\mu_j$ share no edges we have $l(\mu_i+\mu_j)=l(|\mu_i-\mu_j|)\ge l(\mu_i)+l(\mu_j)$ so the associated vectors 
 are orthogonal. 

If the cycles associated to $\mu_i$ 
and $\mu_j$ share edges the length of $l(\mu_i+\mu_j)$ is twice the length of all edges of positive overlap plus the length of all 
edges that are part of one cycle but not the other. The length of $l(|\mu_i-\mu_j|)$ is twice the length of all edges of negative overlap 
plus the length of all edges that are part of one cycle but not the other. Thus $l(\mu_i+\mu_j)-l(|\mu_i-\mu_j|)$ is twice the difference 
of the length of edges of positive overlap and the length of edges of negative overlap.

Therefore the torus is isomorphic to the Albanese torus of the quantum graph by lemma \ref{inner_product}. 
\end{proof}

The complexity of a graph is the number of spanning trees.
\begin{corollary}
 If the quantum graph is equilateral the Bloch spectrum determines the complexity of the graph.
\end{corollary}
\begin{proof}
This follows directly from a theorem in \cite{KotaniSunada00}. For combinatorial graphs the complexity of the graph is given by 
$K(G)=\sqrt{vol(Alb(G))}$. The Albanese torus of an equilateral quantum graph is identical to the Albanese torus of 
the underlying combinatorial graph.
\end{proof}

\begin{remark}
Leaves in a graph are invisible to the homology. So it is not clear whether the entire Bloch spectrum gives us any more information about them 
than the spectrum of a single Schr\"odinger type operator. There are examples of tree graphs that are isospectral for the standard Laplacian,
 see for example \cite{GutkinSmilansky01}. 
\end{remark}

\section{Determining graph properties}
%%%%%%%%%%%%%%%%%%%%%%%%%%%%%%%%%%%%%%%%%%%%%%%%%%%%%%%%%%%%%%%%%%%%%%%%%%%%%%%%%%%%%%%%%%%%%%%%%%%%%%%%%%%%%%%%%%%%%%%%%%%%%%%

We will now use the information gained in the last section and translate it into graph properties that are determined by the Bloch spectrum.

\begin{remark}
The Albanese torus distinguishes the isospectral examples of van Below in \cite{vonBelow99}. Thus the spectrum of a single Schr\"odinger type 
operator does not determine the Albanese torus. In one of the two graphs  
two periodic orbits of length $3$ can be composed to get a periodic orbit of length $4$. Thus the lattice that corresponds to the Albanese 
torus contains two vectors of length $3$ whose sum has length $4$. In the other graph this is not the case. In particular these two graphs are 
not Bloch isospectral by \ref{determine_Albanese}.
\end{remark}

\subsection{The block structure}
%%%%%%%%%%%%%%%%%%%%%%%%%%%%%%%%%%%%%%%%%%%%%%%%%%%%%%%%%%%%%%%%%%%%%%%%%%%%%%%%%%%%%%%%%%%%%%%%%%%%%%%%%%%%%%%%%%%%%%%%%%%%%%%

\begin{theorem}
\label{block_structure}
 One can recognize the block structure, see definition \ref{define_block}, of a leafless graph from the Bloch spectrum. 
It also determines the dimension of the homology of each block.
\end{theorem}
\begin{proof}
 Pick a minimal set of generators $\mu_1, \ldots, \mu_n$ of the group spanned by the frequencies such that the associated minimal 
periodic orbits are all cycles.
 A cycle is necessarily contained within a single block, see \ref{cycle_single_block}. Declare two generators 
 equivalent if the associated cycles share edges regardless of orientation. This generates an equivalence relation. 
Let $\mathfrak{B}$ be the set of equivalence classes, it corresponds to the set of blocks of $G$, see \ref{define_block}. 
The number of generators in each equivalence class is the dimension of the cohomology of that block.

Let $B_1, B_2 \in \mathfrak{B}$. Let $\{\mu^j_i\}_i$ be the subset of frequencies that is $B_j$, $j=1,2$. Then we can find the distance 
between the two blocks by computing 
\begin{equation*}
 d(B_1,B_2):=\frac{1}{2}\min_{i,i'}\left( l(\mu^1_i + \mu^2_{i'})-l(\mu^1_i)-l(\mu^2_{i'}) \right)
\end{equation*}
That is we compute the distance between any basis cycle in one 
block to any basis cycle in the other and minimize over all pairs of basis cycles in the blocks. This distance is zero if and only if the 
blocks share a vertex.

We will now set up a situation where we can apply lemma \ref{tree_from_leaves}. To do so we need to find out which blocks are leaves in 
the block structure and which ones are inner vertices. We will then cut the block structure into smaller pieces such that 
all blocks are leaves in the smaller pieces.

Whenever we have a triple of blocks satisfying $d(B_2,B_3) > d(B_1,B_2)+d(B_1,B_3)$, 
that is a failure of the triangle inequality, 
we know that $B_1$ has to be an inner vertex in the block structure of $G$. The path between the blocks $B_2$ and $B_3$ has to pass through $B_1$ 
and use some edges within the block $B_1$. Once we have identified a block, say $B_1$, as an inner block we can seperate the remaining 
blocks into groups depending on where the path from the block to $B_1$ is attached on $B_1$. If $d(B_i,B_j) > d(B_1,B_i)+d(B_1,B_j)$ then 
the paths from $B_1$ to $B_i$ and $B_j$ 
are attached 
at different cut vertices of $B_1$, if $d(B_i,B_j) \le d(B_1,B_i)+d(B_1,B_j)$ they are attached at the same cut vertex. 
Within each of these groups the block $B_1$ is a leaf in the block structure. 

Thus we have cut the initial block structure into several smaller pieces each of them 
including $B_1$ and $B_1$ is a leaf in each of them.
 We can repeat this process of identifying an inner block and cutting the block structure into smaller pieces on each of these pieces until 
all the pieces have no inner block vertices. This reduces the 
problem to recovering the block structure of graphs where all blocks are leaves. 

All remaining inner vertices have to be vertices of the initial graph $G$ and thus have degree at least $3$. As $G$ is leafless 
all leaves in the block structure are fat vertices. Hence we can recover the tree 
structure of each of the groups of blocks by using lemma \ref{tree_from_leaves}. We can then find the block structure of the entire graph 
by glueing the pieces together at the inner blocks.
\end{proof}

\subsection{Planarity and dual graphs}
%%%%%%%%%%%%%%%%%%%%%%%%%%%%%%%%%%%%%%%%%%%%%%%%%%%%%%%%%%%%%%%%%%%%%%%%%%%%%%%%%%%%%%%%%%%%%%%%%%%%%%%%%%%%%%%%%%%%%%%%%%%%%%%

\begin{theorem}
\label{planarity}
The Bloch spectrum determines whether or not a graph is planar.
\end{theorem}
\begin{proof}
The homology admits infinitely many bases, but a graph has only finitely many cycles and thus there are only finitely many 
bases consisting of cycles. Thus there are only finitely many minimal sets of generators $\mu_1, \ldots, \mu_n$ of the group spanned by 
the frequencies 
such that the minimal periodic orbits associated to them are all cycles. Given such a basis we can choose for each generator to either 
keep the orientation induced by $\alpha$ or choose the reverse orientation. Denote the basis 
elements with a choice of orientation 
by $\gamma_1=\pm \mu_1, \ldots, \gamma_n=\pm \mu_n$. For any pair of oriented basis elements $\gamma_i$ and $\gamma_j$ we can check 
whether the associated cycles have edges of positive overlap by checking whether $l(|\gamma_i+\gamma_j|)-l(|\gamma_i|)-l(|\gamma_j|)>0$, see 
\ref{check_overlap_sign}. Thus we can check whether the 
$\gamma_1, \ldots, \gamma_n$ correspond to a basis of the homology that consists of oriented cycles having no positive overlap.
The graph is planar if and only if we can find such a basis by theorem \ref{simple=planar}. 
\end{proof}

\begin{remark}
 Planarity is a property that is not determined by the spectrum of a single Schr\"odinger type operator. There is an example of two isospectral 
quantum graphs in \cite{vonBelow99} where one is planar and the other one is not.
\end{remark}

If the graph is planar we will fix a non-positive basis $\gamma_1, \ldots, \gamma_n$ (see definition \ref{non-positive_basis}) coming 
from the frequencies $\mu_1, \ldots, \mu_n$.
We know that the basis elements are the boundaries of the inner faces in a suitable embedding of the graph by lemma 
\ref{positive=facial}. We will use this fact to construct an abstract dual of the graph.

\begin{theorem}
\label{abstract_dual}
 The Bloch spectrum of a planar, 2-connected graph determines a dual of the graph. Thus the Bloch spectrum determines 
planar, 2-connected graphs up to 2-isomorphism (see lemma \ref{2-isomorphic}).
\end{theorem}

Before we show this we need two lemmata.

\begin{lemma}
Let $G$ be planar and $2$-connected. In the embedding where the $\gamma_1, \ldots, \gamma_n$ are the boundaries of the inner faces the 
boundary of the outer face is given by 
\begin{equation*}
\gamma_0:=-\sum_{l=1}^{n}\gamma_l      
\end{equation*}
The sign orients it so that it does not have edges of positive overlap with any of the $\gamma_l$.
\end{lemma}

\begin{lemma}
\label{decomposition}
Let $G$ be planar and $2$-connected. 
Then we can determine the number of edges that any two of the cycles $\gamma_0, \ldots, \gamma_n$ have in common.
\end{lemma}
\begin{proof}
Recall that $\mu_i=|\gamma_i|$. If $l(\mu_i+\mu_j) \ge l(\mu_i)+l(\mu_j)$ then $\gamma_i$ and $\gamma_j$ share no edges. 
We will assume from now on that $l(\mu_i+\mu_j) < l(\mu_i)+l(\mu_j)$. 
Suppose $\gamma_i$ and $\gamma_j$ share $k$ 
edges or single vertices. 
\begin{figure}[h]
\centering
\scalebox{0.6}{\includegraphics{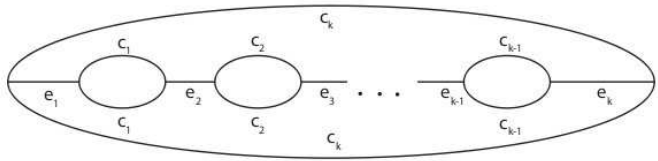}}
\caption{the graph with the two cycles $\gamma_i$ and $\gamma_j$}
\end{figure}

Here $\gamma_i$ and $\gamma_j$ bound the two big faces and the $e_l$ are the edges or single vertices these two cycles share. The remainder 
of the graph is contained inside the small cycles labeled $c_l$, $l=1,\ldots ,k-1$ and outside the big cycle $c_k$.

We can decompose the minimal periodic orbit associated to the frequency $\mu_i+\mu_j$ into several cycles $c_1, \ldots, c_k$ by applying 
lemma \ref{recognize_cycle} repeatedly. These cycles do not share any edges, they 
can share vertices. If the decomposition yields $k$ cycles, then $\gamma_i$ and $\gamma_j$ have $k$ distinct 
components in common. Each of these components is either a single edge or a vertex. Whenever it is a vertex that means that two of the $c_l$ have 
this vertex in common and thus have distance zero from each other. As we can check for any pair $c_l$ and $c_{l'}$ whether 
$l(c_l+c_{l'})=l(c_l)+l(c_{l'})$ we can find all instances where this happens. All remaining common components then must correspond to a common
edge of $\gamma_i$ and $\gamma_j$.
\end{proof}

\begin{remark}
 Lemma \ref{decomposition} is false without the planarity assumption. There exist two cycles $\gamma_1, \gamma_2$ in $K_{3,3}$ that 
share $3$ edges but $\gamma_1 \cup \gamma_2$ is homologous to a single cycle. These two cycles have no edges of positive overlap.
\end{remark}

\begin{proof}
 of theorem \ref{abstract_dual}\\
The cycles $\gamma_0, \ldots, \gamma_n$ are the set of all boundaries of faces in a suitable embedding of the graph. Therefore they are 
the vertices of a geometric dual. By 
lemma \ref{decomposition} we know the number of edges any two of these faces have in common, which corresponds to the number of edges 
between the two vertices in the geometric dual.
\end{proof}

The particular geometric dual we get from this process depends on the non-positive basis we have chosen.

\begin{corollary}
 The Bloch spectrum identifies and determines planar, $3$-connected graphs combinatorially. 
\end{corollary}
\begin{proof}
Note that $3$-connected implies $2$-connected, see definition \ref{k-connected}.
A graph is $2$-connected if its block structure 
consists of a single fat vertex. We have shown that we can identify $2$-connected planar graphs in theorem \ref{planarity}. 
We found a geometric dual of a $2$-connected planar graph in theorem \ref{abstract_dual}. If the dual of the dual is 
$3$-connected it will be unique and therefore isomorphic to the original graph. 
\end{proof}

\section{Determining the edge lengths}
%%%%%%%%%%%%%%%%%%%%%%%%%%%%%%%%%%%%%%%%%%%%%%%%%%%%%%%%%%%%%%%%%%%%%%%%%%%%%%%%%%%%%%%%%%%%%%%%%%%%%%%%%%%%%%%%%%%%%%

In this chapter we will show that we can recover all the edge lengths of a $3$-connected graph if we know the underlying 
combinatorial graph.

The Bloch spectrum only gives us a map from the abstract torus $H^1(G,\R) \slash H^1(G,\Z)$ to the spectra.
By theorem \ref{Bloch_H1} we know the length of the minimal periodic orbit(s) associated to each element in $H_1(G,\Z)$. 
In this section we need a little more, we 
want to associate the lengths we get from the Bloch spectrum with the periodic orbits in the combinatorial graph. 

When we construct a dual graph as in theorem \ref{abstract_dual} we can keep track of the lengths associated to the 
minimal set of generators $\mu_1, \ldots, \mu_n$ of the group spanned by the frequencies. The vertices of the 
dual graph correspond to these frequencies. The vertices of the dual graph then correspond to a set of cycles in the dual of the dual 
that generates the homology. If the graph is $3$-connected the dual of the dual is isomorphic to the original graph so we can associate 
the frequencies and their lengths to the periodic orbits in the graph.

\begin{theorem}
\label{edge_length}
Let $G_0$ be a 3-connected combinatorial graph. Let $G$ be a quantum graph with underlying combinatorial 
graph $G_0$. Then we can recover the length of each edge in $G$ from the Bloch spectrum. That is we can can reconstruct $G$ 
metrically. 
\end{theorem}
\begin{proof}
Given an edge $e$ there are at least 3 disjoint paths that connect its end vertices as $G_0$ is $3$-connected. Thus there are two cycles 
in $G$ that share the edge $e$ and its end vertices but otherwise are disjoint. Denote these two cycles 
by $c_1$ and $c_2$. Denote the closed walk $c_1 \setminus \{e\} \cup (-c_2 \setminus \{e\})$ by $c_3$. Since $c_1$ and $c_2$ are disjoint away 
from $e$ the closed walk $c_3$ is a cycle. Thus the Bloch spectrum determines the lengths 
of $c_1$, $c_2$ and $c_3$ by theorem \ref{Bloch_H1} and we can use them and solve for the length of $e$, 
$2l(e)=l(c_1)+l(c_2)-l(c_3)$.
\end{proof}

\section{Disconnected graphs}
%%%%%%%%%%%%%%%%%%%%%%%%%%%%%%%%%%%%%%%%%%%%%%%%%%%%%%%%%%%%%%%%%%%%%%%%%%%%%%%%%%%%%%%%%%%%%%%%%%%%%%%%%%%%%%%%%%%%%%

If we do not assume that the quantum graph is connected we get a componentwise version of theorem \ref{Bloch_H1}.

\begin{proposition}
Let $G$ be a quantum graph that may or may not be connected. Then the spectrum of the standard Laplacian $\Delta_0$ determines 
the number of connected components. Denote the connected components by $G_1, \ldots, G_k$. 

Given a generic $\alpha$, see definition \ref{generic}, the part of the Bloch spectrum $Spec_{t\alpha}(G)$ for $t\in [0,\varepsilon)$ 
 determines the groups $H_1(G_i, \Z)$ and the length of the minimal periodic orbit(s) of each element in $H_1(G_i, \Z)$ for 
each component $i=1, \ldots, k$. 
\end{proposition}
\begin{proof}
The multiplicity of the eigenvalue $0$ of the standard Laplacian $\Delta_0$ is equal to the number of connected components.

The trace formula in theorem \ref{trace_formula} still holds for disconnected graphs. The two sums over eigenvalues and periodic 
orbits are just unions over the connected components. The total length of the quantum graph is additive and the Euler characteristic 
is well defined for disconnected graphs, too.

Thus we can copy most of the proof of theorem \ref{Bloch_H1} verbatim and read out a set of frequencies from the Bloch spectrum. 
Every frequency we get is associated to a single periodic orbit that belongs to only one of the connected components.
If the sum of two frequencies is a frequency, then these two frequencies belong 
to the same connected component of $G$. If it is not they belong to different connected components.
Thus the set of frequencies (union their negatives and zero) will not form one finitely generated free abelian subgroup of $\R$ that 
is isomorphic to $H_1(G,\Z)$. Instead it will form $k$ disjoint (apart from zero) finitely generated free abelian subgroups 
of $\R$ that are isomorphic to the $H_1(G_i,\Z)$ for $i=1, \ldots, k$. 

We can now assign a length to each frequency the same way as in theorem \ref{Bloch_H1}.
\end{proof}

As all our subsequent theorems are just consequences of theorem \ref{Bloch_H1} they also hold componentwise.

\begin{corollary}
 The theorems \ref{determine_Albanese}, \ref{block_structure}, \ref{planarity}, \ref{abstract_dual} and \ref{edge_length} all hold 
componentwise.
\end{corollary}

\section*{References}

\bibliographystyle{amsalpha}
 \appendix
\addcontentsline{toc}{section}{Literature}
\bibliography{literatur}

\end{document}